\documentclass[11pt]{amsart}
\usepackage{color}
\usepackage{graphicx}
\usepackage{graphics}
\usepackage{amssymb,amsmath,amsthm,epsfig}
\theoremstyle{plain}

\newtheorem{lem}{Lemma}[section]

\newtheorem{prop}[lem]{Proposition}

\newtheorem{rem}[lem]{Remark}
\newtheorem{theorem}[lem]{Theorem}
\newtheorem{lemma}[lem]{Lemma}

\newtheorem{proposition}[lem]{Proposition}

\theoremstyle{definition}

\newtheorem{definition}{Definition}

\theoremstyle{remark}

\newtheorem{remark}{Remark}

\numberwithin{equation}{section}

\def\ffi{\varphi}
\def\eps{\varepsilon}
\def\dst{\displaystyle}

\def\bb{{\mathcal B}}

\def\ff{{\mathcal F}}

\def\qq{{\mathcal Q}}

\newcommand{\NN}{\mathbb N}

\newcommand{\RR}{\mathbb R}

\newcommand{\TT}{\mathbb T}

\newcommand{\R}{\mathbb R}

\newcounter{rea}
\newcommand{\E}{\mathbf E}

\newcounter{rek}
\newcounter{rep}
\DeclareMathOperator{\supp}{supp}

\DeclareMathOperator{\vect}{Span}

\newcommand{\norm}[1]{{\left\|{#1}\right\|}} 
\newcommand{\beq}{\begin{equation}}
\newcommand{\eeq}{\end{equation}}

\newcommand{\abs}[1]{{\left|{#1}\right|}}
\newcommand{\scal}[1]{{\left\langle{#1}\right\rangle}}

\begin{document}

\title[The spectrum of the Sinc  Kernel  Operator]{Non-Asymptotic behaviour of the spectrum of  the Sinc  Kernel  Operator and Related Applications.} 

\author{Aline Bonami}
\address{A.B.\,: F\'ed\'eration Denis-Poisson, MAPMO-UMR 6628, Department of Mathematics, University of Orl\'eans, 45067 Orl\'eans cedex 2, France.}
\email{aline.bonami@univ-orleans.fr}

\author{Philippe Jaming}
\address{P.J.\,: Univ. Bordeaux, IMB, UMR 5251, F-33400 Talence, France.
CNRS, IMB, UMR 5251, F-33400 Talence, France.}
\email{Philippe.Jaming@math.u-bordeaux.fr}

\author{Abderrazek Karoui}
\address{A.K.\,: University of Carthage,
Department of Mathematics, Faculty of Sciences of Bizerte, Tunisia}
\email{abderrazek.karoui@fsb.rnu.tn}

\thanks{This work was supported in part by the French-Tunisian  CMCU 15G 1504 project and the
DGRST  research grant UR 13ES47.}


\begin{abstract}
Prolate spheroidal wave functions have recently attracted a lot of attention in applied harmonic analysis, signal processing and mathematical physics. They are eigenvectors of the Sinc-kernel operator $\mathcal{Q}_c:$ the time- and band-limiting operator. The corresponding eigenvalues
play a key role and it is the aim of this paper to obtain precise non-asymptotic estimates for these eigenvalues, within the three main
regions of the spectrum of $\mathcal{Q}_c$.  This issue  is rarely studied in the literature, while  the asymptotic behaviour of the spectrum of $\mathcal{Q}_c$ has been well established  from the sixties.
As  applications of our non-asymptotic estimates, we first  provide estimates for the constants appearing in  Remez and Tur\`an-Nazarov type concentration inequalities. Then, we give an estimate for the hole probability, associated with a random matrix from the Gaussian Unitary Ensemble (GUE).
\end{abstract}
\maketitle

\section{Introduction}
The aim of this paper is to give precise non-asymptotic estimates to the
eigenvalues associated with the prolate spheroidal wave functions (PSFWs). They are defined as  
the eigenvectors of the time and band-limiting operator given by the sinc-kernel
$$
\qq_c f(x)=\int_{-1}^1\frac{\sin c(y-x)}{\pi(y-x)}f(y)\,\mathrm{d}y.
$$
Here, we recognize  the Sinc convolution kernel which is the reproducing kernel of
the Paley-Wiener space of $c-$bandlimited functions. From another perspective, this kernel is the well-known Dyson Sine kernel which is 
related to the Wigner-Mehta-Dyson universality conjecture about the local statistics of the eigenvalues of random Wigner matrices.
The  eigenvalues $\big(\lambda_n(c)\big)_n$ corresponding to $\qq_c$ play important roles in a wide range of scientific area such as signal 
processing, mathematical-physics, random matrices, numerical analysis including spectral methods, etc. 

The PSFWs have been introduced into signal processing in the work of Landau, Pollak and 
Slepian \cite{LP1,LP2,Slepian1}. They provide an orthonormal basis of the Paley-Wiener space that is optimally
concentrated in the time domain. One striking result in this field is that $\Omega$-band limited functions that have 
their energy concentrated in an interval of length $T$ are well represented by the its expansion restricted to the
first $4\Omega T$ coefficients in that basis. The heuristics of this result is predicted by Shannon's sampling theorem.
Lately, the PSWFs  have been used  in numerical analysis  \cite{XRY}, including the numerical schemes for solving certain PDEs
(see {\it e.g.} \cite{Wang2}).\\

\noindent For random matrices,
Dyson \cite{Dy}, (see also the work of  de Cloizeaux and Mehta \cite{dCM,Me}) showed that the repulsion between eigenvalues of the Gaussian unitary
ensemble (the GUE) could be described asymptotically in terms of the determinantal point process associated with the sine kernel.
In particular, it is well known, see for example \cite{Me, Tao} 
that if $M_n$ is an $n\times n$ random matrix from the GUE and if $A_n=\sqrt{n} M_n$ is the fine scaled  version of $M_n$,
and if $E_2^n(0,c)$ is the probability that there is no eigenvalues of $A_n$ in  an interval of length $\pi c$ and located near 
the centre of the bulk region of the spectrum, then $E_2^n(0,c)$ converges as $n\rightarrow +\infty$ to
the Fredholm determinant 
\begin{equation}
\label{eq:mehta}
E_2(0,c)=\det(Id - \qq_{\pi c/2})=\dst\prod_{j=0}^\infty \Big(1-\lambda_j\big(\pi c/2\big)\Big).
\end{equation}
For more details on the contributions of the PSWFs on these and other applications, the reader is refereed to
\cite{HL} and the relevant references therein.  

Nowadays, there exists a rich literature devoted to the asymptotic behaviour as well as the numerical computation of the eigenvalues 
$\bigl(\lambda_n(c)\bigr)_n,$ to cite but a few
\cite{Bonami-Karoui3, Fuchs, Kuznetsov, Landau1, Slepian3, Widom}. 
These works show that the eigenvalues exhibit three kinds of behaviours, 
defined by the following three distinct regions of the spectrum via a critical index $n_c=\dst\frac{2}{\pi}c$:\\ 

\noindent
$\bullet$ The slow evolution region, where  for $n_c-n \gtrsim \log(c)$,
the change in the $\lambda_n(c)$'s is very slow. For most of the values of these $n,$ we have  $\lambda_n(c)\approx \lambda_0(c)\approx 1$ when $c$ is large enough.\\

\noindent
$\bullet$ The fast decay region, where  for $n-n_c\gtrsim \log(c)$ we have $\lambda_n(c)\to 0$ at a super-exponential speed.\\

\noindent
$\bullet$ The plunge region, which is a transition region between the two previous ones. It is defined 
for those values of $n,$ with  $|n-n_c| \lesssim \log(c).$ Thus the width of this region is  $\approx \log(c)$ for $c$ large.\\

Let us give more details on the literature concerning the asymptotic behaviours of the  $\lambda_n(c)$'s inside the three different 
regions.  Widom in \cite{Widom} has given an asymptotic formula for the super-exponential 
decay rate of the $\lambda_n(c)$. Moreover, Fuchs in \cite{Fuchs} has shown that for fixed  integer $n\geq 0$, the eigenvalue 
$\lambda_n(c)$ converges to $1$ at an optimal exponential rate. Also, Landau and Widom in \cite{Landau0} 
have given a precise asymptotic estimate, valid for $c\gg 1$ of the number of eigenvalues lying in the plunge region. As a consequence
of this estimate, one also has an asymptotic estimate for the eigenvalues decay inside the plunge region, as well as a precise
estimate of the width of this region. More details on these different asymptotic results will be given in the next section.\\
\bigskip

Later on, we study in more  details each of the previous regions. In particular, we give some precise details on the slow, medium and rapid 
changes in the spectrum, along the previous three regions.

\bigskip

\noindent
While the knowledge of the asymptotic behaviour of the $\lambda_n(c)$'s (either with respect to $n$ or with respect to $c$ for fixed $n$) 
is sufficient in some applications, recent applications of the PSWFs require the knowledge of a non-asymptotic behaviour of the $\lambda_n(c)$'s
within the previous regions of the spectrum. This issue is rarely explored in the literature.
Recently, Israel and Osipov in \cite{Arie, Os} provided some non-asymptotic
behaviours of the spectrum in a neighbourhood of the plunge region.  More precisely, in    
\cite{Os}, the author has shown that 
$$
|\Lambda_{\varepsilon} | \leq \frac{2c}{\pi} + K (\log c)^2,
$$
where $K$ is a constant independent of $c$ and $\varepsilon.$  Recently, in \cite{Arie}, 
an improved non-asymptotic estimate of the eigenvalues $\lambda_n(c)$ in the gauge
$$
\Lambda'_{\varepsilon} = \{ k \in \mathbb N : \varepsilon< \lambda_k(c) \leq 1-\varepsilon\},\qquad \varepsilon\in (0,1/2),
$$
has been recently given. More precisely, it is shown in \cite{Arie} that for each $\eta\in (0, 1/2],$
 $$ |\Lambda'_{\varepsilon} | \leq \frac{2c}{\pi} + K_{\eta} \left(\log(\log(c)\cdot \varepsilon^{-1})\right)^{1+\eta}\cdot \log( c \cdot \varepsilon^{-1}),$$ 
for some constant  $K_{\eta},$ depending only on $\eta.$ The present work is a new contribution in the direction of the non-asymptotic  behaviours
 of the spectrum of $\qq_c$ within its  three distinct regions and not only in the neighbourhood of the plunge region.

\smallskip

Let us now describe the results of this paper. The first part of the paper consists in obtaining new and rather sharp estimates of the
decay rate of the $\lambda_n(c)$'s. Our main result in this direction may be summarized as follows:

\medskip

\noindent{\bf Theorem.} {\em Under the previous notation, we have: \\

--- For  any $c >0$ and any  $\dst 0\leq n<\frac{2c}{2.7},$
\begin{equation}
\label{eq:thmain1}
1-\frac{7}{\sqrt{c}}\frac{(2 c)^n}{n!}e^{- c}\leq \lambda_n(c)<1.
\end{equation}

--- For any  $c\geq 22$  and $\eta=0.069,$ we have, for $\dst\frac{2c}{\pi}+\log c+6\leq n\leq c$,
\begin{equation}
\label{eq:thmain2}
\lambda_n(c)\leq\frac{1}{2}\exp\left(-\frac{\eta\left(n-\frac{2c}{\pi}\right)}{\log c+5}\right).
\end{equation}

--- For any $c >0$ and any ${\displaystyle n\geq \max\Big(\frac{e c}{2},2\Big),}$ 
\begin{equation}
\label{eq:thmain3}
\lambda_n(c)\leq\exp\left(-(2n+1)\log\frac{2}{ec}(n+1)\right).
\end{equation}
}

\medskip

The actual results are slightly more precise. The result \eqref{eq:thmain1} is obtained by exploiting the min-max
theorem tested on Hermite functions. The second form of the min-max theorem also allows to derive \eqref{eq:thmain3}.
A more precise estimate can be obtained at the price of much more involved computations by exploiting our previous work
\cite{Bonami-Karoui3}. This method can be pushed almost to the plunge region and leads to \eqref{eq:thmain2}.

\smallskip

In the second part of this paper, we exploit these bounds in two directions. First, for $\epsilon >0,$ we 
consider the best constant $\Gamma_2(n,\eps)$ in Remez type inequalities: for every polynomial $P$
of degree at most $n$,
$$
\int_{[-\eps/2,\eps/2]}|P(e^{it})|^2\,\mbox{d}t\geq\Gamma_2(n,\eps)\int_{[-\pi,\pi]}|P(e^{it})|^2\,\mbox{d}t.
$$
We show that $2\lambda_n(c)\geq \Gamma_2\bigl(n,2c/(n+1)\bigr)$ from which we deduce a new upper bound for $\Gamma_2$. 
Also,  we give a lower bound for the constant $A$ appearing in the Tur\`an-Nazarov concentration inequality that relates 
for $0\leq q\leq 2,$ the two quantities $\|P\|_{L^q(E)}$ and $\|P\|_{L^q(\TT)},$ where $P$ is a trigonometric polynomial over the torus
and $E$ is a measurable subset of $\TT.$

\smallskip

Finally, we use our lower bound of the $\lambda_n(c)$'s, given by \eqref{eq:thmain1}, and give an estimate 
of the hole probability $E_2(0,c),$ given by \eqref{eq:mehta}.

\medskip

The remaining of the paper is organized as follows. In the next section, we introduce the necessary notation and mathematical preliminaries that we will use. Section 3 is devoted to results that we obtain through the min-max principle while Section 4 is devoted
to the precise estimates of $\lambda_n(c)$ built on our previous work.
Finally Section 5 is devoted to the applications of those estimates.

\section{Mathematical Preliminaries}

\subsection{Notation}  

As precise constants are computed, it is important to be precise with the normalizations used in this paper. 
The Fourier transform is defined for $f\in L^1(\R)\cap L^2(\R)$ by
$$
\widehat{f}(\xi):=\ff[f](\xi):=\int_\R f(x) e^{-ix\xi}\,\mbox{d}x.
$$
Parseval's theorem then reads
$$
\int_\R|f(x)|^2\,\mbox{d}x=\frac{1}{2\pi}\int_\R|\widehat{f}(\xi)|^2\,\mbox{d}\xi
$$
and Fourier inversion when $\widehat{f}\in L^1(\R)$ reads
$$
f(x)=\frac{1}{2\pi}\int_\R \widehat{f}(\xi) e^{ix\xi}\,\mbox{d}\xi:=\ff^{-1}[\widehat{f}](x).
$$
The operators $\ff$ and $\ff^{-1}$ are then extended to $L^2(\R)$ in the usual way.
We then consider the time-band limiting operator $\qq_c$, $c>0$ as 
$$
\qq_c f=\mathbf{1}_{[-1,1]}\ff\left[\mathbf{1}_{[-c,c]}\ff^{-1}[\mathbf{1}_{[-1,1]}f]\right]
$$
when seen as an operator $L^2(\R)\to L^2(\R)$
or simply $\qq_c f=\ff^{-1}\bigl[\mathbf{1}_{[-c,c]}\ff[f]\bigr]$ when seen as an operator $L^2([-1,1])\to L^2([-1,1])$.
Note also that, applying Parseval's identity, $\norm{\qq_c f}_{L^2(-1,1)}\leq \norm{f}_{L^2(-1,1)}$. 
It is not hard to see that this operator is given as a Hilbert-Schmidt operator via a sinc-kernel
$$
\qq_c f(x)=\int_{-1}^1\frac{\sin c(y-x)}{\pi(y-x)}f(y)\,\mbox{d}y.
$$
We define the finite Fourier transform $\ff_c\,:L^2([-1,1])\to L^2([-1,1])$ as
$$
\ff_c[f](x)=\int_{-1}^1 f(y) e^{-icxy}\,\mbox{d}y=\ff[\mathbf{1}_{[-1,1]}f](cx)
$$
and then $\ff_c^*\,:L^2([-1,1])\to L^2([-1,1])$ is given by
$$
\ff_c^*[g](x)=\int_{-1}^1 g(y)e^{icxy}\,\mbox{d}y.
$$
A simple computation then shows that $\qq_c=\frac{2c}{\pi}\ff_c^*\ff_c$. This shows that $\qq_c$ is self-adjoint, positive and compact. 
Its eigenvalues $(\lambda_n(c))_{n\geq 0}$ are then arranged as follows,
$$
1\geq \lambda_0(c)>\lambda_1(c)>\cdots>\lambda_n(c)>\cdots
$$

In \cite{Slepian1}, D. Slepian and H. Pollack have pointed out the  crucial property that the 
operator $\qq_c$ commutes with a Sturm-Liouville operator $\mathcal L_c$,   defined on $C^2([-1,1])$ by
\begin{equation}\label{eq1.0}
\mathcal L_c(\psi)=-\frac{d}{d\, x}\left[(1-x^2)\frac{d\psi}{d\,x}\right]+c^2 x^2\psi.
\end{equation}
Consequently, both operators have the same infinite and countable set of eigenfunctions $\psi_{n,c}(\cdot),\,\, n\geq 0,$ called prolate spheroidal wave functions (PSWFs). The eigenvalues of the differential and integral operators $\mathcal L_c$ and $\qq_c$ are denoted by $\chi_n(c)$ and $\lambda_n(c),$ so that for any integer $n\geq 0,$ we have 
\begin{equation}\label{Eq1.0}
\mathcal L_c(\psi_{n,c})(x)=\chi_n(c) \psi_{n,c}(x),\qquad \qq_c (\psi_{n,c})(x)=\lambda_n(c) \psi_{n,c}(x),\quad x\in [-1,1].
\end{equation}
Note  that the eigenvalues $\chi_n(c)$ and $\lambda_n(c)$ are simple. Moreover, a straightforward application of the variational formulation of the eigenvalues of the self-adjoint operator $\mathcal L_c$ shows that 
\begin{equation}\label{Eqq1.0}
n(n+1)=\chi_n(0)\leq \chi_n(c) \leq n(n+1)+c^2,\qquad n\geq 0.
\end{equation}
For more details, see \cite{Slepian1}. Also, note that in the special case $c=0,$ the $\psi_{n,c}$ are reduced to the well-known  Legendre polynomials $P_n$.

Since $\lambda_n(c)$ is the $n+1$-th eigenvalue of $\qq_c$, the min-max theorem asserts that for $I=[-1,1],$ we have
\begin{equation}\label{eq:minmax1}
\lambda_{n}(c) =\min_{S_n} \max_{f\in S_n^{\perp}} \frac{\scal{ \mathcal Q_c f, f}_{L^2(I)}}{\|f\|^2_{L^2(I)}}=
\min_{S_n} \max_{f\in S_n^{\perp}} \frac{2c}{\pi}\frac{\scal{\mathcal F_c f, \mathcal F_c f}_{L^2(I)}}{\|f\|^2_{L^2(I)}},
\end{equation}
where $S_n$ is a set of $n$-dimensional subspaces of $L^2(I)$. On the other hand, for $f\in L^2(\R)$, 
\begin{eqnarray*}
\scal{\qq_c f,f}_{L^2(I)}&=&\scal{\ff\bigl[\mathbf{1}_{(-c,c)}\ff^{-1}[\mathbf{1}_{(-1,1)}f]\bigr],\mathbf{1}_{(-1,1)}f}_{L^2(\R)}\\
&=&\frac{1}{2\pi}\scal{\mathbf{1}_{(-c,c)}\ff^{-1}[\mathbf{1}_{(-1,1)}f],\ff^{-1}[\mathbf{1}_{(-1,1)}f]}_{L^2(\R)}\\
&=&\frac{\norm{\ff^{-1}[\mathbf{1}_{(-1,1)}f]}^2_{L^2(-c,c)}}{2\pi}.
\end{eqnarray*}
Thus, if we identify $S_m$ with the set of $m$-dimensional subspaces of the subspace of $L^2(\R)$ consisting
of functions with support in $[-1,1]$, then the min-max Theorem implies that
$$
\lambda_n(c)=\sup_{V\in S_{n+1}}\min_{f\in V\setminus\{0\}}\frac{\norm{\ff^{-1}[\mathbf{1}_{(-1,1)}f]}^2_{L^2(-c,c)}}{2\pi\norm{\mathbf{1}_{(-1,1)}f}_{L^2(\R)}^2}
=\sup_{V\in \mathcal{V}_{n+1}(1)}\min_{g\in V\setminus\{0\}}\frac{\norm{g}^2_{L^2(-c,c)}}{\norm{g}_{L^2(\R)}^2}
$$
where $\mathcal{V}_m(a)$ is the set of $m$-dimensional subspaces of the Paley-Wiener space of $a-$bandlimited functions
$$
\mathcal B_a:=\{g=\ff^{-1}[f]\,:\ f\in L^2(\R),\ \mathrm{Supp}\,f\subset[-a,a]\}.
$$
Let $g\in \bb_1$ and $g_c(x)=c^{1/2}g(cx)$. Then $\norm{g_c}_{L^2(\R)}=\norm{g}_{L^2(\R)}$,
$\norm{g_c}_{L^2(-1,1)}=\norm{g}_{L^2(-c,c)}$ and $g_c=\ff^{-1}[f_c]$ where $f_c(y)=c^{-1/2}f(y/c)\in L^2(\R)$
with support in $[-c,c]$, thus $g_c\in \bb_c$. We thus get
\begin{equation}\label{eq:minmax2}
\lambda_n(c)=\sup_{V\in\mathcal{V}_{n+1}(c)}\inf_{f\in V\setminus\{0\}}\frac{\|f\|^2_{L^2(-1, +1)}}{\|f\|^2_{L^2(\R)}},\quad n\geq 0.
\end{equation}

In particular, $\lambda_0(c)$ is the largest constant such that the inequality
$$
C\norm{f}_{L^2(\R)}^2\leq \norm{f}_{L^2(-1,1)}^2
$$
holds for some $c$-band limited function $f$.

A compacity argument shows that there exists $f_0\in\bb_c\setminus\{0\}$ such that
$\lambda_0(c)\norm{f_0}_{L^2(\R)}^2\leq \norm{f_0}_{L^2(-1,1)}^2$. But then $\lambda_0(c)=1$ would
imply that $f_0$ is supported in $[-1,1]$ which is 
impossible since $f_0\not=0$ is band-limited thus an entire function. As a consequence, $\lambda_0(c) < 1.$

\subsection{Previous estimates on $\lambda_n(c)$}

In the sequel, we denote by $\mathbf K$ and  $\mathbf E,$  the    elliptic  integrals of the first and second kind, given respectively, by 
\begin{equation}\label{elliptic_integrals}
\mathbf K(t)=\int_0^1 \frac{ \mathrm{d}s}{\sqrt{(1-s^2)(1-s^2 t^2)}}\qquad 
\mathbf E(t)=\int_0^1 \sqrt{\frac{1- s^2 t^2}{1-s^2}}\,\mathrm{d}s, \quad  0\leq s\leq 1.
   \end{equation} 
We first recall from the literature, some decay results of the 
sequence of the eigenvalues $(\lambda_n(c))_n.$ The first asymptotic behaviour of these eigenvalues has been given 
by Widom, see \cite{Widom},
\begin{equation}\label{Widom_asymptotic}
\lambda_n(c)\sim  \left ( \frac {e c} {4(n+\frac 12)} \right)^{2n+1}=\lambda_n^W(c).
\end{equation}
An important result that describes the asymptotic behaviour of the spectrum inside its slow evolution region and plunge region, is given 
by the following  Landau-Widom asymptotic eigenvalues counting formula. More precisely, for $\varepsilon \in (0,1/2)$, let 
$$
\Lambda_{\varepsilon} = \{ k \in \mathbb N : \lambda_k(c) \geq \varepsilon\},\qquad \varepsilon\in (0,1/2).
$$
Then, in \cite{Landau0}, the authors have shown that for $c\gg 1,$ we have  
\begin{equation}\label{Landau_Estimate}
|\Lambda_{\varepsilon} | = \frac{2c}{\pi} + \frac{1}{\pi^2} \log\left(\frac{1-\varepsilon}{\varepsilon}\right) \log(c)+o(\log c).
\end{equation}
On the other hand,Fuchs in \cite{Fuchs} has shown that in the slow evolution region of the spectrum, for a fixed integer
$n\geq 0,$ the eigenvalue $\lambda_n(c)$ converges exponentially to one, with respect to $c$. More precisely, Fuchs asymptotic formula states that 
\begin{equation}\label{Fuchs_asymptotic}
1-\lambda_n(c) \sim 4\sqrt{\pi} \frac{8^n}{n!} c^{n+1/2} e^{-2c},\qquad c\gg 1.
\end{equation}
 
Recently, in \cite{Bonami-Karoui3}, the authors have given a  precise explicit approximation formula for $\lambda_n(c),$ which is valid for $\frac{\pi n}{2}-c$ larger than some multiple of $\ln n$. This formula gives also the precise asymptotic super-exponential decay rate of the spectrum of $\mathcal Q_c.$
More precisely, let us denote by  $\Phi$ the inverse of the function $\Psi\,:t\mapsto \frac{t}{\E(t)}$, then this approximation formula is given by 
\begin{equation}\label{decay22}
\lambda_n(c) \sim \widetilde { \lambda_n}(c) =  \frac{1}{2} \exp\left(-\frac{\pi^2(n+\frac 12)}{2} \int_{\Phi\left(\frac{2c}{\pi (n+\frac 12)}\right)}^1 \frac{1}{t(\mathbf E(t))^2}\,\mathrm{d}t \right).
\end{equation}
As a consequence, from Corollary 3 of \cite{Bonami-Karoui3}, there exist three constants $\delta_1\geq 1, \delta_2, \delta_3, \geq 0$ such that, for $n\geq 3$ and $c\leq \frac{\pi n}2$,
\begin{equation}\label{best}
A(n, c)^{-1}\left(\frac{ec}{2(2n+1)}\right)^{2n+1}\leq  \lambda_n(c)\leq A(n,c)\left(\frac{ec}{2(2n+1)}\right)^{2n+1}.
\end{equation}
with 
$$
A(n, c)=\delta_1 n^{\delta_2}\left(\frac c{c+1}\right)^{-\delta_3}e^{+\frac {\pi^2}4 \frac{c^2}{n}}.
$$
Moreover from Proposition 4 of \cite{Jaming-Karoui-Spektor}, we have the following  lower decay rate of the $\lambda_{n}(c),$
\begin{equation}\label{lower_decay1}
\lambda_n(c)\geq7\left(1-\frac{2c}{n\pi}\right)^2\left(\frac{c}{7\pi n}\right)^{2n-1}, \dst n>\frac{2}{\pi}c.
\end{equation}
This lower decay rate is obtained by combining the Min-Max characterization of the $\lambda_n(c),$ and the Tur\`an-Nazarov
concentration inequality \cite{Nazarov}. 

\section{Min-Max technique and  non-asymptotic behaviour of the spectrum of the Sinc  kernel operator.}

In this section, we first give a simple proof of a fairly tight upper  bound of  the super-exponential decay rate of the eigenvalues of the Sinc-kernel operator.  This simple proof is based on the use of the Min-Max characterisation of the eigenvalues of a self-adjoint compact operator. Then, this last technique is used to provide us with a lower bound of the
eigenvalue $\lambda_n(c),$ with $n\geq 0$ not too large.  We should mention that this Min-Max theorem based method has the advantage to be relatively simple to apply. Nonetheless, when used to estimate the super-exponential decay rate of the $\lambda_n(c),$ this method  has the drawback to be valid only sufficiently far from the plunge region around $n_c=\frac{2c}{\pi}.$ This is given by the following theorem.

\begin{theorem} Let $c>0$ be a real number,  then for any integer $n \geq \max\left(2,\frac{ e c}{2}\right),$ we have 
\begin{equation}\label{Eq3.1}
\lambda_{n}(c) \leq \exp\left[-(2n+1)\log
\left(\frac{2}{ec}(n+1)\right)\right].
\end{equation}
\end{theorem} 

\begin{proof} Recall from \eqref{eq:minmax1} that
$$
\lambda_{n}(c) =\min_{S_n} \max_{f\in S_n^{\perp}} \frac{\scal{ \mathcal Q_c f, f}}{\|f\|^2_{L^2(I)}}=
\min_{S_n} \max_{f\in S_n^{\perp}} \frac{2c}{\pi}\frac{\scal{\mathcal F_c f, \mathcal F_c f}}{\|f\|^2_{L^2(I)}},
$$
where the $S_n$ are $n-$dimensional subspaces of $L^2(I)$.

Recall that the Legendre polynomials normalized by $P_n(1)=1$ are defined by the Rodrigues formula
$$
P_n(x)=\frac{1}{2^nn!}\frac{\mathrm{d^n}}{\mathrm{d}x^n} (1-x^2)^n
$$
and that, if we renormalize them by $\widetilde P_k= \sqrt{k+1/2} P_k$,
then 
$\widetilde P_n, \, n\geq 0$ form an orthonormal basis of $L^2(I)$.
We will take the following special choice of $S_n$ in \eqref{eq:minmax1}:
$$
S_n =\mbox{Span}\{\widetilde P_0,\ldots,\widetilde P_{n-1}\}.
$$
Hence, if $f\in S_n^{\perp},$ then ${\displaystyle f=\sum_{k=n}^{\infty} a_k \widetilde P_k.}$ 
We may assume that $\|f\|^2_{L^2(I)}=1,$ so that ${\displaystyle \sum_{k=n}^{\infty} a_k^2 =1.}$ 
It follows that
$$
\mathcal{F}_c f=\sum_{k=n}^{\infty} \sqrt{k+1/2}a_k \mathcal{F}_cP_k(x)
$$
and
\begin{eqnarray*}
\norm{\mathcal{F}_c f}&=&\sum_{k=n}^{\infty} \sqrt{k+1/2}|a_k|\norm{\mathcal{F}_cP_k}\leq
\left(\sum_{k=n}^{\infty}|a_k|^2\right)^{1/2}\left(\sum_{k=n}^{\infty}(k+1/2)\norm{\mathcal{F}_cP_k}^2\right)^{1/2}\\
&\leq&\left(\sum_{k=n}^{\infty}(k+1/2)\norm{\mathcal{F}_cP_k}^2\right)^{1/2}.
\end{eqnarray*}
Combining this with \eqref{eq:minmax1}, we obtain
\begin{equation}
\label{eq:upbnd}
\lambda_{n}(c)\leq\frac{2c}{\pi}\sum_{k=n}^{\infty}(k+1/2)  \norm{\mathcal{F}_cP_k}^2.
\end{equation}

On the other hand, it is known that, see for example \cite{NIST}
\begin{equation}\label{Eq3.4}
\mathcal{F}_c P_k(x)=\int_{-1}^1 e^{icxy} P_k(y)\, dy = i^k \sqrt{\frac{2\pi}{cx}} J_{k+\frac{1}{2}} (cx),\quad x\in I,
\end{equation}
where $J_{\alpha}$ is the Bessel function of the first type and order $\alpha>-1.$ 

Further, the Bessel function $J_{\alpha}$ has the following fast decay with respect to the parameter $\alpha,$
$$
|J_{\alpha}(z)|\leq \frac{\left|\frac{z}{2}\right|^{\alpha}}{\Gamma(\alpha+1)},
$$
where $\Gamma(\cdot)$ is the Gamma function ({\it see e.g.} \cite{NIST}).
Combining this with the classical estimate of the Gamma function 
$$
\Gamma(x+1)\geq\sqrt{2e} \left(\frac{x+\frac{1}{2}}{e}\right)^{x+\frac{1}{2}}
$$
we obtain
$$
|J_{k+1/2}(cx)|^2\leq \frac{1}{2(k+1)}\left(\frac{ec}{2(k+1)}\right)^{2k+1} |x|^{2k+1}.
$$
From this, we deduce that
$$
\norm{\mathcal{F}_c P_k}^2\leq \frac{\pi}{c(k+1)}\left(\frac{ec}{2(k+1)}\right)^{2k+1}\int_{-1}^1 |x|^{2k}\,\mathrm{d}x
=\frac{2\pi}{c(k+1)(2k+1)}\left(\frac{ec}{2(k+1)}\right)^{2k+1}.
$$
Injecting this into \eqref{eq:upbnd}, we obtain
$$
\lambda_n(c)\leq
2\sum_{k=n}^{\infty}\frac{1}{(k+1)}\left(\frac{ec}{2(k+1)}\right)^{2k+1}
\leq \frac{2}{(n+1)}\sum_{k=n}^{\infty}\left(\frac{ec}{2(k+1)}\right)^{2k+1}.$$
We recall that $ec\leq n$ so that for $k\geq n$,
$$
\left(\frac{ec}{2(k+1)}\right)^{2k+1}\leq\left(\frac{ec}{2(n+1)}\right)^{2n+1}\left(\frac{n}{2(n+1)}\right)^{2(k-n)}.
$$
So we have
$$
\lambda_n(c)\leq \frac{2}{n+1}\left(\frac{ec}{2(n+1)}\right)^{2n+1}\sum_{k=n}^{\infty}\left(\frac{n}{2(n+1)}\right)^{2(k-n)}
\leq \left(\frac{ec}{2(n+1)}\right)^{2n+1}
$$
as claimed.
\end{proof}

In the remaining of this section, we use a second version of the Min-Max theorem to 
get some fairly precise lower bound for the first eigenvalues $\lambda_0(c),\ldots,\lambda_n(c)$ when
$n\leq\dst\frac{c}{2.7}$.

\begin{theorem} For $n$ an integer such that $0\leq n\leq\dst\frac{c}{2.7}$, 
we have
\begin{equation}
\label{lower_bound}
\lambda_n(c)\geq 1-\frac{7}{\sqrt{c}}\frac{(2 c)^n}{n!}e^{- c}.
\end{equation}
\end{theorem}

\begin{proof} 
In order to prove the theorem, we need some notation. Let $\Pi_c$ be the projection 
from $L^2(\R)$ on $\bb_c$ {\it i.e.} $\Pi_c f=\ff^{-1}\bigl[\mathbf{1}_{(-c,c)}\ff [f]\bigr]$, 
and $\widetilde{\Pi}_c=I-\Pi_c$ the projection on the orthogonal of $\bb_c$.
Let $H_n$ be the $n$-th Hermite polynomial given by
$$
H_n(t)=(-1)^ne^{t^2}\frac{\mathrm{d}^n}{\mathrm{d}t^n}\bigl(e^{-t^2}\bigr)
$$
and $\ffi_n(t)=\alpha_nH_n(t)e^{-t^2/2}$ with $\alpha_n=\dst\frac{1}{\pi^{1/4} 2^{n/2}\sqrt{n!}}$. 
It is then well known that $(\ffi_n)_{n\in\NN}$ is an orthonormal basis of $L^2(\R)$ of eigenfunctions of the normalized Fourier transform,
$\frac{1}{\sqrt{2\pi}}\ff[\ffi_n]=i^{-n}\ffi_n$.
Further, we check the following probably known result
\begin{equation}
\label{Hermite_bound}
|H_n(x)|\leq 2^n |x|^n,\quad\forall\, |x|\geq \sqrt{2n-2}.
\end{equation}
From the parity of the Hermite polynomials, it suffices to check the previous inequality for $x\geq \sqrt{2n-2}.$ 
It is well known that if $z_i$ denotes the largest zero of the Hermite polynomial $H_i(x),$ then
$$
\sqrt{2n-2} \geq z_n > z_{n-1} >\cdots >z_1=0.
$$
We prove by induction that $|H_n(x)|\leq 2^n |x|^n$ for $x>z_n$. This is valid for $H_1(t)=2t$. Then, assuming that it is valid for $k-1$ and since $H_k'(x)= 2 k H_{k-1}(x),$ the inequality for $H_k$ follows immediately from the fact that 
$\dst H_k(x)=2k\int_{z_k}^x H'_{k-1}(t) dt.$

Hence, by using \eqref{Hermite_bound}, one gets
\begin{equation}\label{lem:estgaus}
|\ffi_n(t)|\leq\frac{2^{n/2}}{\pi^{1/4} \sqrt{n!}}|t|^ne^{-t^2/2},\quad \forall\, |t| \geq \sqrt{2n-2}.
\end{equation}
Next, we note  that for  $\alpha>1$ and $a>\sqrt{\alpha-1}$, we have
\begin{equation}
\label{eq:estgaus}
\int_a^{+\infty} t^\alpha e^{-t^2}\,\mathrm{d}t\leq a^{\alpha-1}e^{-a^2}.
\end{equation}
In fact, under the previous conditions on $a$ and $\alpha,$ the function
$t\rightarrow t^{\alpha-1} e^{-t^2/2}$ is decreasing on $(a,+\infty)$ and consequently 
$$ 
\int_a^{+\infty} t^\alpha e^{-t^2}\,\mathrm{d}t\leq a^{\alpha-1}e^{-a^2/2}\int_a^{+\infty} t e^{-t^2/2}\,\mathrm{d}t=a^{\alpha-1}e^{-a^2}.$$
Applying  \eqref{lem:estgaus}, we deduce that for $k\leq \dst\frac{a^2-1}{2}$,
\begin{equation}
\label{eq:estffin0}
\int_{|t|\geq a} \ffi_k^2(t)\,\mathrm{d}t
\leq \frac{2^{k+1}}{\pi^{1/2} k!}\int_a^{+\infty}t^{2 k}e^{-t^2}\,\mathrm{d}t
\leq \frac{2^{k+1}}{\pi^{1/2} k!}a^{2k-1}e^{-a^2}.
\end{equation}
Next, using the fact that $\int_{\R}\ffi_k^2(t)\,\mathrm{d}t=1$, we get
\begin{equation}
\label{eq:estffin1}
\int_{|t|\leq a}\ffi_k(t)^2\,\mathrm{d}t\geq 1-\frac{2}{\pi^{1/2} a}\frac{(2a^2)^k}{k!}e^{-a^2}.
\end{equation}

Further, we define $\ffi_n^{(c)}(t)=c^{1/4}\ffi_n(\sqrt{c}t)$ and note that $(\ffi_n^{(c)})_{n\in\NN}$ is still
an orthonormal basis of $L^2(\R)$. Let $V_n(c)=\vect\{\Pi_c\ffi_0^{(c)},\ldots,\Pi_c\ffi_n^{(c)}\}$.
From \eqref{eq:minmax2} we know that, 
\begin{eqnarray}
\lambda_n(c)&\geq& \inf\left\{\frac{\norm{f}_{L^2(-1,1)}^2}{\norm{f}_{L^2(\R)}^2}\,: f\in V_n(c)\right\}\nonumber\\
&=&\inf\left\{\frac{\norm{f}_{L^2(-1,1)}^2}{\norm{f}_{L^2(\R)}^2}\,: f=\sum_{j=0}^n\gamma_j\Pi_c\ffi_j^{(c)}\,,\ \sum_{j=0}^n|\gamma_j|^2=1\right\}.
\label{eq:lambda0}
\end{eqnarray}
Now, take a sequence $(\gamma_j)_{j=0,\ldots,n}$ with $\dst\sum_{j=0}^n|\gamma_j|^2=1$.
Define $F=\dst\sum_{k=0}^n\gamma_k\ffi_k^{(c)}$ and $f=\Pi_cF$.

As $\Pi_c$ is an orthonormal projection on $\bb_c,$ which is a closed subspace of $L^2(\R),$ then we have 
$$
\norm{f}_{L^2(\R)}\leq \norm{F}_{L^2(\R)}=\sum_{k=0}^n|\gamma_j|^2=1.
$$
and
$$
\norm{F}_{L^2(\R)}^2 = \norm{F-\widetilde{\Pi}_cF}_{L^2(\R)}^2+\norm{\widetilde{\Pi}_cF}_{L^2(\R)}^2,\quad \widetilde{\Pi}_c=I-\Pi_c.
$$
Consequently, we have
\begin{eqnarray}
\norm{f}_{L^2(-1,1)}^2&=&\norm{\Pi_c F}_{L^2(-1,1)}^2=\norm{F-\widetilde{\Pi}_cF}_{L^2(-1,1)}^2\nonumber\\
&\geq& \norm{F-\widetilde{\Pi}_cF}_{L^2(\R)}^2-\norm{F-\widetilde{\Pi}_cF}_{L^2\bigl(\R\setminus (-1,1)\bigr)}^2\nonumber\\
&\geq&\norm{F}_{L^2(\R)}^2-\norm{\widetilde{\Pi}_cF}_{L^2(\R)}^2-\norm{F-\widetilde{\Pi}_cF}_{L^2\bigl(\R\setminus (-1,1)\bigr)}^2\nonumber\\
&\geq & 1-3\norm{\widetilde{\Pi}_cF}_{L^2(\R)}^2-2\norm{F}_{L^2\bigl(\R\setminus(-1,1)\bigr)}^2. \label{eq:up}
\end{eqnarray}

We now estimate the error terms. They will be obtained through the same computation, which we do first for 
$\norm{F}_{L^2\bigl(\R\setminus(-1,1)\bigr)}^2.$ We have the following inequalities.
\begin{eqnarray*}
\norm{F}_{L^2\bigl(\R\setminus(-1,1)\bigr)}^2&=&\int_{|x|>1}\abs{\sum_{j=0}^n\gamma_j\ffi_j^{(c)}(x)}^2\,\mathrm{d}x
= \int_{|x|> c^{1/2}}\abs{\sum_{j=0}^n\gamma_j\ffi_j(x)}^2\,\mathrm{d}x\\
&\leq&\sum_{k=0}^n\int_{|x|> c^{1/2}}|\ffi_k(t)|^2\,\mathrm{d}t.
\end{eqnarray*}
We have used the inequality of Cauchy-Schwarz and the fact that $\dst\sum_{k=0}^n|\gamma_k|^2=1$. By using \eqref{lem:estgaus}, we obtain the inequality

\begin{equation}
\norm{F}_{L^2\bigl(\R\setminus(-1,1)\bigr)}^2\leq \frac{2}{\pi^{1/2} c^{1/2}}\sum_{k=0}^n\frac{(2 c)^k}{k!}e^{- c}. 
\end{equation}
Now, when $n\leq \frac{ c}{2.7}$,  we get 
\begin{equation}\label{Sum_etimate}
\sum_{k=0}^n\frac{(2 c)^k}{k!}\leq \frac{(2 c)^n}{n!}\sum_{k=0}^\infty\left(\frac{n}{2c}\right)^k\leq  1.23
\frac{(2 c)^n}{n!}.
\end{equation}

Finally, we have
\begin{equation}
\label{eq:EstFL2-111}
\norm{F}_{L^2\bigl(\R\setminus(-1,1)\bigr)}^2\leq \frac{2.46}{\pi^{1/2} c^{1/2}}\frac{(2 c)^n}{n!}e^{- c}.
\end{equation}

Next, 
$$
\norm{\tilde\Pi_c F}_{L^2(\R)}^2=\frac 1{2\pi}\norm{\ff[\tilde\Pi_c F]}_{L^2(\R)}^2=\frac 1{2\pi}\int_{|\xi|\geq c}|\ff[F](\xi)|^2\,\mathrm{d}\xi.
$$
But
\begin{eqnarray*}
\frac 1{\sqrt{2\pi}}\ff[F](\xi)&=& \frac 1{\sqrt{2\pi}}\sum_{k=0}^n\gamma_k\ff[\ffi_k^{(c}](\xi)\\
&=&\sum_{k=0}^n c^{-1/4}\gamma_k \frac 1{\sqrt{2\pi}}\ff[\ffi_k](c^{-1/2}\xi)\\
&=&\sum_{k=0}^n c^{-1/4}i^{-k}\gamma_k \ffi_k(c^{-1/2}\xi).
\end{eqnarray*}
As we have done for the bound of $\norm{F}_{L^2\bigl(\R\setminus(-1,1)\bigr)}^2$, given by \eqref{eq:EstFL2-111}, one gets
\begin{eqnarray*}
\frac 1{{2\pi}}\norm{\ff[\tilde\Pi_c F]}_{L^2(\R)}^2&=&\int_{|\xi|\geq c}\abs{\sum_{k=0}^n c^{-1/4}i^{-k}\gamma_k \ffi_k(c^{-1/2}\xi)}^2\,\mathrm{d}\xi\\
&=&\int_{|\eta|\geq c^{1/2}}\abs{\sum_{k=0}^n i^{-k}\gamma_k \ffi_k(\eta)}^2\,\mathrm{d}\eta\\
&\leq&\sum_{k=0}^n\int_{|x|> c^{1/2}}|\ffi_k(t)|^2\,\mathrm{d}t \leq \frac{2.46 (2c)^ne^{-c}}{ \pi^{1/2}\sqrt{c} n!}.
\end{eqnarray*}
Parseval's equality thus implies that
\begin{equation}\label{eq:EstFL2-12}
\norm{\tilde\Pi_c F}_{L^2(\R)}^2=\frac{1}{2\pi}\norm{\ff[\tilde\Pi_c F]}_{L^2(\R)}^2
\leq 2.46\frac{ (2c)^ne^{-c}}{\pi^{1/2}\sqrt{c} n!}.
\end{equation}
By combining \eqref{eq:EstFL2-12}, \eqref{eq:EstFL2-111} and \eqref{eq:up}, we get
\begin{eqnarray*}
\norm{f}_{L^2(-1,1)}^2&\geq&1- \frac{12.3}{ \pi^{1/2} \sqrt{c}}\frac{(2 c)^n}{n!}e^{- c}\geq 1 - \frac{7}{\sqrt{c}} \frac{(2 c)^n}{n!}e^{- c}.
\end{eqnarray*}
Finally, as $\norm{f}_{L^2(\R)}\leq 1$, we get the desired inequality \eqref{lower_bound}.
\end{proof}

\begin{remark}
The previous  theorem gives a better estimate than Landau's Theorem only when the bound below, given by \eqref{lower_bound} is larger than $1/2$.
We will content ourselves to verify that the bound below given for $\lambda_n(c)$  has a positive sign.
This is given by the following lemma.
\end{remark}

\begin{lemma}
For any $c\geq 4$ and any integer ${\displaystyle 0\leq n\leq  \frac{c}{2.7}, }$ we have 
\begin{equation}\label{positivity}
 1-\frac{7}{\sqrt{c}}\frac{(2 c)^n}{n!}e^{- c}>0.
 \end{equation}
\end{lemma}

\begin{proof}
We first rewrite the inequality \eqref{positivity} as follows
\begin{equation}\label{Ineq1}
n \log(2c)-c < \frac{1}{2} \log c - \log 7 +\log(n!).
\end{equation}
Moreover, since $n!= \Gamma(n+1)$ and for $x\geq 0,$ $\Gamma(x+1)\geq \sqrt{2e} \left(\frac{x+\frac{1}{2}}{e}\right)^{x+\frac{1}{2}},$
then the previous inequality is satisfied whenever 
$$ 
-c \leq n \log\left(\frac{n+1/2}{2c}\right)-n +\frac{1}{2}\bigl(\log(c(n+1/2))-1+\log(2 e)-2\log 7\bigr).
$$
Since $n\geq 0$ and $c\geq 4,$ then elementary computations show that this last inequality holds true whenever 
$$ 
n \log\left(\frac{n+1/2}{2c}\right)-n > -c.
$$
It is elementary to  verify this inequality for $n\leq \frac{c}{2.7},$ using monotonicity of 
${\displaystyle x\rightarrow x \log\left(\frac{x+1/2}{2c}\right)-x }$  on $[0,c].$ We have chosen the constant $2.7$ for this purpose.
\end{proof}

\section{The Sinc kernel operator: Non-Asymptotic decay rate of the spectrum.}

In this section, we need to introduce a new function related to the elliptic function $\E$ defined in \eqref{elliptic_integrals}:
let $\Phi(\cdot)$ the function defined on $[0,1]$ as the inverse of the function $t\mapsto \Psi(t)=\frac{t}{\E(t)}$. 
We will show that  the rapid decay starts shortly after  $\frac{2c}{\pi} + \log c .$
Also,  by using Landau's asymptotic formula for the number of eigenvalues $\lambda_n(c)$ that are greater than 
a threshold $0<\varepsilon < 1/2$, we check that  our asymptotic decay rate in the previous part of the plunge region is optimal 
with respect to the parameter $c.$
We give three different statements depending on how $n$ compares to $c$:
  
\begin{theorem}
\label{decay}
Let $c\geq 22$ be a real number, $\eta= 0.069$ and $\eta'=0.12$.

--- For any integer $n$ such that
$\frac{2c}{\pi} +\log c +6\leq n\leq  c,$ we have the inequality
\begin{equation}\label{decay1}
\lambda_n(c)\leq \frac{1}{2} \exp\left[-\frac{\eta(n-\frac{2c}{\pi})}{\log(c) +5}\right].
\end{equation}

--- Let $\frac{2}{\pi}\leq\delta\leq 1$. Under the assumption that 
$n\geq \max(\frac{2c}{\pi} +\log c +6, \frac {2 c}{\delta\pi}),$ 
we have the estimate
\begin{equation}\label{decay11}
\lambda_n(c)\leq \frac{1}{2} \Big(\Phi(\delta)\Big)^{0.069 n}.  
\end{equation}
In particular,
$$
\lambda_n(n)\leq \frac{1}{2} e^{-0.015 n}.
$$

--- For any any integer $n\geq c$ such that 
$\frac{\pi}2(n -\log n -9)\geq  c,$ we have 
\begin{equation}\label{ddecay1} 
\lambda_n(c) \leq \exp \left[-\eta' ( n -c)\right].
\end{equation}
\end{theorem}

\begin{remark} The assumption that $c\geq 22$ guarantees that there exists an integer $n$  such that $ \frac{2c}{\pi} +\log c +6 \leq n\leq c$. This condition can be relaxed with the right hand replaced by $2c$, for instance, which allows to have such an estimate for smaller values of $c$. But the constant $\eta$ is smaller.
\end{remark} 

\begin{remark} In view of \eqref{decay1} one may replace the condition on $n$ by the condition  $n\geq \frac{2c}{\pi} +\delta(\log c +6)$. This is possible as long as $\delta>2/\pi$. As one expects, the constant $\eta$ will increase with $\delta$. This constant is certainly far from being optimal. But the decrease in $n$ is optimal.
\end{remark}

  \begin{proof}
  We use a series of results of \cite{Bonami-Karoui2, Bonami-Karoui3}. 
   We recall that from Theorem 2 of \cite{Bonami-Karoui2}, we have 
   \begin{equation}\label{Eq2.1}
   \psi_{n,c}(1) \geq \left(\chi_n(c)\right)^{1/4}\sqrt{\frac{\pi}{2 \mathbf K(\sqrt{q_n})}}\left(1-\frac{3}{(1-q_n) \sqrt{\chi_n(c)}}\right),\qquad q_n= \frac{c^2}{\chi_n(c)},
   \end{equation}
   whenever the condition
  $$
   (1-q_n) \sqrt{\chi_n(c)} \geq 4
  $$ 
  is satisfied. Here, we recall that $\chi_n(c)$ is as given by \eqref{Eq1.0} and satisfying \eqref{Eqq1.0}.
  Also, from \cite[Lemma 3]{Bonami-Karoui3}, the previous condition   is satisfied whenever 
 $$
  n > \frac{2c}{\pi}+\frac{2}{\pi}\left(\log(n) +6\right),
  $$     
  or equivalently, for any 
  \begin{equation}\label{Eqq2.7}
  0 < c < c^*_n := \frac{\pi n}{2}-\log(n) - 6.
  \end{equation}
So, under this condition, we have
\begin{equation}\label{Eq2.3}
\left(\psi_{n,c}(1)\right)^2 \geq \frac{\pi}{32}
\frac{\sqrt{\chi_n(c)}}{\mathbf K(\sqrt{q_n})}=\frac{\pi}{32}\frac{c}{\sqrt{q_n}\mathbf K(\sqrt{q_n})}.
\end{equation}
 
Next, recall that $\Phi(\cdot)$ denotes the inverse of the function $t\mapsto \Psi(t)=\frac{t}{\E(t)}$. 
From \cite[Theorem 1]{Bonami-Karoui2}, we get
\begin{equation}
\label{Phi}
\Phi\left(\frac{2c}{\pi (n+1)}\right) < \sqrt{q_n} < \Phi\left(\frac{2c}{\pi n}\right).
\end{equation}
Since the function $\mathbf K(\cdot)$ is increasing on $[0,1)$, we deduce from \eqref{Eq2.3} that
\begin{equation}\label{Eq2.6}
 \frac{\left(\psi_{n,c}(1)\right)^2}{c} \geq \frac{\pi}{32}\frac{1}{ \Phi\left(\frac{2c}{\pi n}\right)\mathbf K\left( \Phi\left(\frac{2c}{\pi n}\right)\right)}.
\end{equation}
 Note that from \cite{Landau1}, we have ${\displaystyle \lambda_n(c^*_n) \leq \frac{1}{2}.}$ Moreover, from the differential
equation
$$
\partial_{\tau} \log(\lambda_n(\tau)) = 2  \frac{\left(\psi_{n,\tau}(1)\right)^2}{\tau}
$$
and by substituting $c$ by $\tau$ in \eqref{Eq2.6}, one gets for $0<c < c^*_n,$
\begin{eqnarray}
\lambda_n(c)&\leq&  \frac{1}{2}\exp\left( -2 \int_c^{c^*_n} \frac{(\psi_{n,\tau}(1))^2}{\tau}  \mathrm{d} \tau\right)\nonumber\\
&\leq& \frac{1}{2} \exp\left( -\frac{\pi}{16} \int_c^{c^*_n} \frac{ \mathrm{d} \tau}{\Phi\left(\frac{2\tau}{\pi n}\right)\mathbf K\left( \Phi\left(\frac{2\tau}{\pi n}\right)\right)}\right)
:=  \frac{1}{2}\exp\left(-\frac{\pi}{16} I_{1,n} \right).
\label{decay_eigenvalues}
\end{eqnarray}
At this point we fix $n$ and $c\geq 22$ such that $n\geq \frac{2c}{\pi} +\log c +6.$ Because of the fact that there exists such an $n$ which is bounded by $c$, for all such values of $n$ we have the inequality $c<c_n^*$, so that we can use the previous formula.

To estimate the integral $I_{1,n},$ we use a computational technique that we have developed in \cite{Bonami-Karoui3}. More precisely, we consider the substitution ${\displaystyle t= \Phi\left(\frac{2\tau}{\pi n}\right)},$ that is 
${\displaystyle \Psi(t)= \frac{2\tau}{\pi n}.}$ Since ${\displaystyle \Psi'(t)= \frac{\mathbf K(t)}{\big(\mathbf E(t)\big)^2},}$
see \cite{Bonami-Karoui3}, then we have 
$$
I_{1,n} =\frac{\pi n}{2}\int_{\Phi\left(\frac{2c}{\pi n}\right)}^{\Phi\left(\frac{2c^*_n}{\pi n}\right)} \frac{ \mathrm{d} t}{t \big(\mathbf E(t)\big)^2}.
$$
We have immediately the bound below
\begin{equation}
\label{Eq2.8}I_{1, n}\geq \frac{\pi n}{2} \frac{\Phi\left(\frac{2c^*_n}{\pi n}\right)-\Phi\left(\frac{2c}{\pi n}\right)}{\Big(\mathbf{E}(\Phi\left(\frac{2c}{\pi n}\right))\Big)^2}.\end{equation}
We can replace the denominator by the constant $\mathbf E (\Phi(2/\pi))^{-2}.$ 

Next,  note that $\mathbf{K}$ is increasing and $\mathbf{E}$ is decreasing, thus $\Psi'$ is increasing and therefore $\Psi$ is convex and $\Phi$ is concave. Hence, if $s<s^*<1$, using the fact that $\Phi(1)=1$, we have
$$
\frac{\Phi(s^*)-\Phi(s)}{s^*-s}\geq \frac{1-\Phi(s)}{1-s}.
$$
Up to now we did not use the assumption on $n$, which we do now. We have chosen $n$ sufficiently far from $c^*_n$ compared 
to the distance between $c_n^*$ and the cutoff value $\frac{2c}{\pi n}$ so that, when 
${\displaystyle s^*=\frac{2c^*_n}{\pi n},\,\, s=\frac{2c}{\pi n},}$ one  has the inequality
$$
\frac{1-s^*}{1-s}\leq 2/\pi.
$$
So the previous inequality can be replaced by 
\begin{equation}\label{Eq2.9a}
\Phi(s^*)-\Phi(s)\geq (1-\frac 2 \pi)(1- \Phi(s)).
\end{equation}
It remains to prove that 
\begin{equation}\label{final}
1-\Phi\left(\frac{2c}{\pi n}\right)\geq \frac{1-\frac{2c}{\pi n}}{\log c +5}.
\end{equation}
Assume that $x:= 1-\frac{1-\frac{2c}{\pi n}}{\log n +5},$ so that \eqref{final} can also be written as
\begin{equation}\label{final-psi}
1-\Psi(x)\leq 1-\frac{2c}{\pi n}.
\end{equation}
We then use Inequality (25) of \cite{Bonami-Karoui2}, which we recall here:
$$
\mathbf E(x)-1\leq (1-x)\left(\log\left(\frac{1}{1-x}\right)+4\right).
$$
It follows that
\begin{equation}\label{psi} 1-\Psi(x)\leq (1-x)\left(\log\left(\frac{1}{1-x}\right)+5\right).\end{equation}
We conclude for \eqref{final-psi} by using  the elementary inequality
$$
\log \left(\frac{1}{1-x}\right)=\log\left(\frac{n(\log c+5)}{n-\frac{2c}{\pi}}\right)\leq \log n,
$$
which is a consequence of the assumption that $n-\frac{2c}{\pi}\geq\log c+5 .$

We have obtained \eqref{decay1} with the constant $\eta= \frac{\pi^2}{32}\times (1-\frac 2 \pi )\times \Big(\mathbf E (\Phi(2/\pi))\Big)^{-2}.$  
The numerical constant given in the statement is obtained from the approximations $\Phi(2/\pi)\approx 0.8$ and $\mathbf E (\Phi(2/\pi))\approx 1.276.$

\bigskip

The proof of \eqref{decay11} is a slight modification of the last one. Instead of \eqref{Eq2.8} we write the a priori better estimate
\begin{equation}
\label{Eq2.8bis}I_{1, n}\geq \frac{\pi n}{2} \frac{\log \left(\Phi\left(\frac{2c^*_n}{\pi n}\right)\right)-\log \left(\Phi\left(\frac{2c}{\pi n}\right)\right)}{\Big(\mathbf{E}(\Phi\left(\frac{2c}{\pi n}\right))\Big)^2}.
\end{equation}
The function $-\log (\Phi(\cdot))$ is convex since $\Phi$ is concave and as before the numerator is bounded below by $(1-\frac{2}{\pi})(-\log \left(\Phi\left(\frac{2c}{\pi n}\right)\right).$ But $-\log \left(\Phi\left(\frac{2c}{\pi n}\right)\right)\geq -\log \Phi(\delta). $ We conclude at once.

\bigskip

Let us finally prove \eqref{ddecay1}. First, the condition $\frac{\pi}2(n -\log n -9)\geq  c$ is necessary to be able to use Theorem 1 in \cite{Bonami-Karoui2}.  Next, we now consider integrals from $c$ to $n$. We use the same notations for the new quantities involved. In particular \eqref{Eq2.8bis} is replaced by 
$$
I_{1, n}\geq \frac{\pi n}{2} \frac{\log \left(\Phi\left(\frac{2}{\pi}\right)\right)
-\log \left(\Phi\left(\frac{2c}{\pi n}\right)\right)}{\Big(\mathbf{E}(\Phi\left(\frac{2c}{\pi n}\right))\Big)^2}.
$$

The denominator can only be bounded by $\pi^2/4$. For the numerator,  we use again the convexity of the function $-\log (\Phi(\cdot))$. It follows that
$$
\frac{\log \left(\Phi\left(\frac{2}{\pi}\right)\right)
-\log \left(\Phi\left(\frac{2c}{\pi n}\right)\right)}{\frac{2}{\pi }-\frac{2c}{\pi n}}
\geq \frac {-\log \left(\Phi\left(\frac{2}{\pi }\right)\right)}{1-\frac{2}{\pi }}.
$$
The constant $\eta'$ may be taken equal to ${\displaystyle \frac{-\pi^2 \log(0.8)}{16(\pi-2)}\approx 0.12}.$ One can conclude easily for the validity of \eqref{ddecay1}.
\end{proof}

\begin{remark} Methods given in the proof of this theorem may be used for other corollaries of the estimates of \cite{Bonami-Karoui3}. The same method can be used to find bounds below of eigenvalues, as in  \cite{Bonami-Karoui3}. It is one way to see that results are  optimal.
\end{remark}

We should mention that  a decay estimate which is similar, but weaker than the one given by \eqref{decay1} can be obtained
from the non-asymptotic estimate of the cardinal of the subset $\Lambda_{\varepsilon}$ of $\mathbb N,$ given by 
$$
\Lambda_{\varepsilon} = \{ k \in \mathbb N : \lambda_k(c) \geq \varepsilon\},\qquad \varepsilon\in (0,1/2).
$$
In \cite{Os}, the author has shown that 
$$
|\Lambda_{\varepsilon} | \leq \frac{2c}{\pi} + K (\log c)^2 \log\Big(\frac{1}{\varepsilon}\Big)= \beta(\varepsilon),
$$
where $K$ is a constant independent of $c$ and $\varepsilon.$ Since the eigenvalues $\lambda_n(c)$ are arranged in the decreasing order, then the previous
inequality implies that ${\displaystyle \lambda_{\beta(\varepsilon)}(c) \leq \varepsilon,}$
or equivalently 
\begin{equation}\label{decay3}
\lambda_n(c) \leq \beta^{-1}(n)=\exp\left(-\frac{n-\frac{2c}{\pi}}{K \big(\log(c)\big)^2}\right),\quad n\geq \frac{2c}{\pi}+ K (\log 2)\big(\log(c)\big)^2 .
\end{equation}
Note that for sufficiently large $c$ and comparing to our non-asymptotic decay rate, given by \eqref{decay1}, the previous decay rate is weaker
because of the power $2$ on $\log(c)$.  Note that the optimal upper bound for the cardinal of the set
$\Lambda_{\varepsilon}$ is given by \eqref{Landau_Estimate}. 
By comparing \eqref{decay1} and \eqref{Landau_Estimate}, one concludes that our non-asymptotic decay estimate of the $\lambda_n(c)$ given by 
\eqref{decay1} is optimal, in the considered part of the plunge region of the spectrum.

\section{Applications}

In this paragraph, we describe two applications of our non-asymptotic estimates of the spectrum of the Sinc-kernel operator. The first application is related to the estimate of constants, associated with  two concentration inequalities. The second application concerns the  probability of a hole in the spectrum of a fine scaled random matrix from the GUE.

\subsection{Estimates of constants in Remez and Tur\`an-Nazarov concentration inequalities.}

We first recall that for $f$ a $2\pi-$periodic function, we have
$$
\norm{f}_{L^2(\TT)}=\left(\frac{1}{2\pi}\int_{-\pi}^{\pi}|f(t)|^2\,\mbox{d}t\right)^{1/2}.
$$
We consider the $n+1$-th eigenvalue $\lambda_n(c)$ for $n>0$. According to the min-max principle, we have  
$$
\lambda_n(c)=\sup_{V\in\mathcal{V}_n(c)}\inf_{f\in V\setminus\{0\}}\frac{\|f\|^2_{L^2(-1, +1)}}{\|f\|^2_{L^2(\R)}},
$$
where $\mathcal{V}_n(c)$ is the set of $n+1$-dimensional subspaces of $\mathcal B_c$.
By invariance by dilation and modulation, for every $M>0$ and every interval $I$ of length $|I|=2Mc$,
this quantity is equal to 
$$
\sup_{V\in\mathcal{V}_n(I)}\inf_{f\in V\setminus\{0\}}\frac{\|f\|^2_{L^2(-1/M, +1/M)}}{\|f\|^2_{L^2(\R)}},
$$
where $\mathcal{V}_n(I)$ is the set of $n+1$-dimensional subspaces of $\mathcal B_I:=\{f\in L^2(\R)\,:\supp\widehat{f}\subset I\}$. 
Now choose $M$ so that $2Mc \geq n+1$. Take $\ffi$ to be a fixed $L^2$ function such that $\supp\widehat{\ffi}\subset[-1/2,1/2]$
and let $V_{n,\ffi}=\{P(e^{it})\ffi\,:P\in\mathcal{P}_{n}\}$ where $\mathcal{P}_{n}$ is the set of all polynomials of degree at most $n$.
Note that, if $f\in V_{n,\ffi}$ then $f=\dst\sum_{k=0}^{n}\widehat{P}(k)e^{ikt}\ffi(t)$ so that
$$
\widehat{f}(\xi)=\sum_{k=0}^{n}\widehat{P}(k)\widehat{\ffi}(\xi-k).
$$
In particular, $\supp\widehat{f}\subset[-1/2,1/2+n]$, an interval of length $n+1=2Mc$.
Further,
\begin{multline*}
\norm{f}_{L^2(\RR)}^2=\frac{1}{2\pi}\|\widehat{f}\|_{L^2(\R)}^2
=\frac{1}{2\pi}\sum_{k=0}^{n}|\widehat{P}(k)|^2\|\widehat{\ffi}\|_{L^2(\RR)}^2\\
=\norm{P}_{L^2(\TT)}^2\norm{\ffi}_{L^2(\RR)}^2=\frac{1}{2\pi} \norm{P(e^{it})}_{L^2(-\pi,\pi)}^2\norm{\ffi}_{L^2(\RR)}^2.
\end{multline*}
On the other hand
\begin{multline*}
\norm{f}_{L^2(-1/M,1/M)}^2=\int_{-1/M}^{1/M}\abs{\sum_{k=0}^{n}\widehat{P}(k)e^{i kt}}^2|\ffi(t)|^2\,\mbox{d}t\\
\geq \min_{[-1/M,1/M]}|\ffi(t)|^2\norm{P}_{L^2(-1/M,1/M)}^2.
\end{multline*}
We choose $\ffi$ by setting $\widehat{\ffi}(\xi)=\mathbf{1}_{[-1/2,1/2]}\cos \pi\xi$ so that
$$
\|\ffi\|_{L^2(\mathbb R)}^2=\frac{1}{2\pi}\int_{-1/2}^{1/2}\cos^2\pi\xi\,\mbox{d}\xi=\frac{1}{4\pi}
$$
and
\begin{eqnarray*}
\ffi(x)&=&\frac{1}{4\pi}\int_{-1/2}^{1/2} e^{i (\pi+x)\xi}+e^{i(-\pi +x)\xi}\,\mbox{d}\xi
=\frac{1}{2\pi}\frac{\sin(\pi+x)/2}{\pi +x}+\frac{1}{2\pi}\frac{\sin(-\pi +x)/2}{-\pi +x}\\
&=&\begin{cases}\dst\frac{\cos x/2}{\pi^2-x^2}&\mbox{if }x\not=\pm\pi\\[2pt]
\dst\frac{1}{4\pi}&\mbox{if }x=\pm\pi\end{cases}.
\end{eqnarray*}
Next, let $M\geq\dst\frac{1}{\pi}$, {\it i.e.} $n+1\geq\frac{2}{\pi}c$. Then, for $x\in[-1/M,1/M]$, $\dst \ffi(x)\geq \ffi(1/M)\geq \ffi(\pi)=\frac{1}{4\pi}$.

Therefore, for any Taylor polynomial $P$ of degree at most $n$,
\begin{equation}\label{Ineq4.1}
2 \lambda_n(M c)\geq \frac{\norm{P(e^{it})}_{L^2(-1/M,1/M)}^2}{\norm{P(e^{it})}_{L^2(-\pi,\pi)}^2}.
\end{equation}

\begin{definition}
Let $\Gamma_2(n, \varepsilon)$ be the best constant in the Remez type inequality
$$
\int_{I}|P(e^{it})|^2\,\mbox{d}t\geq \Gamma_2(n, \varepsilon)\int_0^{2\pi}|P(e^{it})|^2\,\mbox{d}t,
$$
for any $P$ is a Taylor polynomial of degree $n$ and  any $I$ interval of length $\varepsilon$.
\end{definition}

We have just shown the following:

\begin{prop}
Let $c>0$ and $n\dst\geq\dst\frac{2c}{\pi}-1$ be an integer, then
$$
2\, \lambda_n(c)\geq \Gamma_2\left(n, 4\frac{c}{n+1}\right)
$$
or, equivalently, for $0<\varepsilon<2\pi$
$$
\Gamma_2\left(n, \eps\right)\leq 2\, \lambda_n\left(\frac{n+1}{4}\eps\right).
$$
In particular, for $0<\varepsilon<4$ and $n\geq 33$,
$$
\Gamma_2\left(n, \eps\right)\leq  \exp\left(\frac{-n}{32} \log\frac{4}{\eps}\right).
$$
\end{prop}

\begin{rem}
In the particular case where $M=\frac{1}{\pi}$ in the inequality \eqref{Ineq4.1} and since the Remez constant is  then equal to $1$ in this case, one concludes that 
\begin{equation}
\label{Ineqq4.1}
\lambda_{n+1}\left(\frac{\pi}{2}(n+1)\right)\geq \frac{1}{2},
\end{equation}
which is the  Landau's lower bound for  ${\displaystyle \lambda_{n}\left(\frac{\pi}{2} n\right).}$
\end{rem}
This is related to the following Tur\`an-Nazarov inequality, see \cite{Nazarov}. Let $\TT$ be the unit circle and let $\mu$ be the Lebesgue measure on $\TT,$ normalized so that $\mu(\TT)=1,$ then for every $0\leq q \leq 2,$ every 
trigonometric polynomial 
$$
{\displaystyle P(z)=\sum_{k=1}^{n+1} a_k z^{\alpha_k},\quad a_k\in \mathbb C,\quad z\in \TT,}
$$
and every measurable subset $E\subset \TT,$ with $\mu(E)\geq \frac{1}{3},$ we have
\begin{equation}
\label{Ineq_Turan}
\| P \|_{L^q(E)} \geq e^{-A \, n\,  \mu(\TT\setminus E)}  \| P\|_{L^q(\TT)}.
\end{equation}
Here, $A$ is a constant independent of $q,$ $E$ and $n.$ It has been mentioned in \cite{Nazarov} that if moreover, 
the measurable subset $E$ is not too large in the sense that ${\displaystyle \mu(E)\leq 1-\frac{\log n}{n},}$ 
then the previous inequality holds also for $q>2.$ 

To  find a bound above for $\gamma_2(n, \varepsilon)=\sqrt{\Gamma_2(n,\varepsilon)},$ where 
$\Gamma_2(n,\varepsilon)$ is given by the previous definition, 
it suffices to note that from our previous analysis, we have for $\varepsilon= 2 M^{-1}\leq \frac{8}{e},$
$$
\gamma_2(n,\varepsilon)\leq  \sqrt{2} \sqrt{\lambda_{n+1}\left(\frac{n+1}{2}\varepsilon\right)}.
$$ 

But from \eqref{Eq3.1}, with ${\displaystyle c=\frac{n+1}{4} \varepsilon,}$ one gets 
\begin{equation}\label{bound_gamma}
\gamma_2(n,\varepsilon)\leq \sqrt{2} \exp\left(-(n+1/2) \log\Big(\frac{8}{e\varepsilon}\Big)\right).
\end{equation}

We compare the  inequality \eqref{bound_gamma} with \eqref{Ineq_Turan} with the special cases, $q=2,$
$$
E=E_M=\{ e^{it},\,\, t\in (-1/M, 1/M)\},\quad \frac{1}{\pi}\leq M\leq \frac{3}{\pi},
$$
is an arc of $\TT.$ The upper bound on $M$ is fixed by the requirement that 
${\displaystyle \mu(E)=\frac{\varepsilon}{2\pi}=\frac{1}{\pi M}\geq 1/3.}$
In this case, we have 
$$
\mu(\TT\setminus  E)= 1-\frac{\varepsilon}{2\pi} =1-\frac{2}{\pi M}.
$$
Consequently, the constant $A$ appearing in \eqref{Ineq_Turan} satisfies the inequality
$$
-A n \left(1-\frac{\varepsilon}{2\pi}\right)\leq \log(\sqrt{2})-(n+1/2) \log\left(\frac{8}{e\varepsilon}\right),\quad 
\forall\,\,\, \frac{2\pi}{3}\leq \varepsilon \leq 2\pi.
$$
By straightforward computations, the previous inequality holds true whenever $A\geq \frac{1}{3}$ and $n\geq 2.$

\subsection{An estimate of the hole probability in the spectrum of a random matrix from the GUE}

We first recall that an $n\times n$ random Hermitian matrix $M_n=\big[X_{jk}\big]_{1\leq j,k\leq n}$ 
is said to belong to the GUE, if  the $X_{jk}$ are i.i.d. random variables with $X_{kj}=\overline{X_{jk}},$ $X_{jj}\sim N(0,1)$ 
and $Re(X_{jk})\sim N(0,1/2),\ Im(X_{jk}) \sim N(0,1/2),$ if $j\neq k.$ Following the notation used in \cite{Tao}, 
the coarse-scaled and the fine-scaled versions of the matrix $M_n$ from the GUE, are defined by 
${\displaystyle W_n= \frac{1}{\sqrt{n}} M_n,\,\, \, A_n = \sqrt{n} M_n.}$
The famous Wigner's semi-circle law states that the empirical measure
${\displaystyle \frac{1}{n} \sum_{i=1}^n \delta_{\lambda_i(W_n)}}$ 
converges weakly in distribution to the semi-circle distribution function, given by 
${\displaystyle \rho_{sc}(x)=\frac{1}{2\pi} \sqrt{4-x^2} {\bf 1}_{[-2,2]}(x).}$
The bulk region of the spectrum of $A_n,$ the fine-scaled version of $M_n,$ corresponds to the eigenvalues of $A_n$ 
which are close to $n u$ for a given $u\in (-2,2).$ It is well known, see for example \cite{Tao} 
that for $a,b\in \mathbb R,$ the  probability the interval
$\Big[ n u +\frac{a}{\rho_{sc}(u)}, n u +\frac{b}{\rho_{sc}(u) }\Big]$ does not contain any eigenvalues of the matrix $A_n=\sqrt{n} M_n,$ converges 
as $n\rightarrow +\infty$ to the Fredholm determinant
$$
\det(Id-\qq_{s})=\dst\prod_{n=0}^\infty \big(1-\lambda_n(s) \big),\quad s= \frac{\pi}{2}(b-a).
$$
In particular, in the centre of the bulk region, that is for $u=0$ and $\rho_{sc}(0)=1/\pi,$ the previous limiting probability that an interval
of length $\pi c$ has no eigenvalue is given by  the following infinite product.
$$ 
E_2(0,c) = \prod_{n=0}^{\infty} \big(1- \lambda_n(\frac{\pi}{2}c)\big).
$$
To get an upper bound for  the previous probability, we first use the estimate \eqref{lower_bound} and write
$$
1- \lambda_k\big(\frac{\pi}{2}c\big) \leq 7\sqrt{\frac{2}{\pi c}} \frac{(\pi c)^k}{k!} e^{-\frac{\pi}{2}c}.
$$
Next, for $ \eta> 1,$ we consider the finite product 
\begin{equation}\label{Eq1}
 F_{\eta}=\prod_{k=0}^{\eta-1} 7\sqrt{\frac{2}{\pi c}} \frac{(\pi c)^k}{k!} e^{-\frac{\pi}{2}c}= e^{-\frac{\pi}{2}c \eta} \left(7\sqrt{\frac{2}{\pi c}}\right)^{\eta} A_{\eta},\qquad  A_{\eta}=  \prod_{k=0}^{\eta-1} \frac{(\pi c)^k}{k!}.
\end{equation}
To bound the quantity $A_{\eta},$ we note that $k!=\Gamma(k+1)\geq \sqrt{2e} \left(\frac{k +\frac{1}{2}}{e}\right)^{k+\frac{1}{2}},$ so that we have
\begin{equation}\label{Eq2}
 A_{\eta} \leq (2e)^{-\eta/2} \prod_{k=0}^{\eta-1} \left(\frac{\pi e c}{k+\frac{1}{2}}\right)^k \sqrt{\frac{e}{k+\frac{1}{2}}}= (2e)^{-\eta/2} B_{\eta}.
\end{equation}
Note that 
$$
\log B_{\eta}=\frac{\eta (\eta-1)}{2} \log(\pi e c)+\frac{\eta}{2}-\sum_{k=0}^{\eta-1} \big(k+\frac{1}{2}\big)\log\big(k+\frac{1}{2}\big).
$$
Moreover, since
$$
\sum_{k=0}^{\eta-1} \big(k+\frac{1}{2}\big)\log\big(k+\frac{1}{2}\big)\geq \int_{\frac{1}{2}}^{\eta-\frac{1}{2}} x \log x \, dx 
-\frac{1}{2}\log 2,
$$
then straightforward computations show that 
$$
\log B_{\eta} \leq  \frac{1}{2} n(\eta-1) \log\left(\frac{\pi e c}{\eta-\frac{1}{2}}\right)+\frac{1}{4} \big(\eta-\frac{1}{2}\big)^2 +\frac{\eta}{2}-\frac{1}{7} \log\big(\eta-\frac{1}{2}\big).
$$
In particular, for the convenient  choice of 
\begin{equation}\label{Eq3}
\eta=\eta_{\gamma}=\frac{\pi}{2}\frac{c}{\gamma} +\frac{1}{2},\quad \gamma= 2.7,
\end{equation}
one gets
\begin{equation}\label{Eq4}
\log B_{\eta}\leq \frac{1}{2} \left(\frac{\pi^2}{4}\frac{c^2}{\gamma^2}-\frac{1}{4}\right)\log(2 e \gamma)+\frac{\pi^2}{16 \gamma^2} c^2 
+\frac{\pi}{4 \gamma} c -\frac{1}{7} \log\left(\frac{\pi c}{4 \gamma}\right)+\frac{1}{4}. 
\end{equation}
On the other, we already know that ${\displaystyle \lambda_n(\frac{\pi}{2}c)\geq \frac{1}{2},}$ for all $n\leq c.$ Consequently, we have
\begin{equation}\label{Eq4-2}
G_{\eta_{\gamma}}=\prod_{n=n_{\gamma}}^{c} \big(1- \lambda_n(\frac{\pi}{2}c)\big)\leq \exp\left(-c\big(1-\frac{\pi}{2\gamma}\big)
 \log 2\right)=\exp(-c\,  a_{\gamma}).
\end{equation}
By combining \eqref{Eq1}--\eqref{Eq4-2} and taking into account that for all $k\geq 0,$  $ 1- \lambda_k\big(\frac{\pi}{2}c\big)\leq 1,$ one can easily check that
\begin{equation}\label{Eq5}
E_2(0,c)\leq  F_{\eta_{\gamma}} G_{\eta_{\gamma}} \leq e^{-\alpha c^2} \exp\left(-\beta c\Big(\log(\sqrt{\pi e c}/7)-1/2+\frac{a_{\gamma}}{\beta}\Big)-\frac{1}{7}\log(\beta c)\right),
\end{equation}
where
\begin{equation}\label{Eq6}
\alpha=\frac{\pi^2}{4\gamma}-\frac{\pi^2}{8\gamma^2}\log(2 e \gamma) -\frac{\pi^2}{16 \gamma}\approx 0.375,\quad  \beta =\frac{\pi}{2\gamma}\approx 0.586,\quad  a_{\gamma}=1-\frac{\pi}{2\gamma}.
\end{equation}
In particular, for $c\geq 5,$ we have $-\beta c\Big(\log(\sqrt{\pi e c}/7)-1/2+\frac{a_{\gamma}}{\beta}\Big)-\frac{1}{7}\log(\beta c)\leq 0,$ so that 
$E_2(0,c) \leq e^{-\alpha c^2},\,\, \alpha\approx 0.374.$
We have just proved the following proposition.

\begin{proposition}
Let  $c\geq 5,$ then  a bound above of  the probability $E_2(0,c)$ is given by 
\begin{equation}\label{bound_above1}
E_2(0,c)\leq    e^{-\alpha c^2},\quad  \alpha \approx  0.374.
\end{equation}
\end{proposition}

Note that an   asymptotic expression of   the probability $E_2(0,c),$ for $c\gg 1,$ has been 
given in \cite{dCM} 
\begin{equation}\label{bound_above2}
E_2(0,c)\simeq  \kappa \exp\left(-\frac{\pi^2}{8} c^2  -\frac{1}{4}\log( c)\right),
\end{equation}
for some positive constant  $\kappa.$  Our proposed non-asymptotic bound above of $E_{2}(0,c)$ is not optimal comparing to the previous asymptotic one,
nonetheless, it is a bound above that is valid even for intervals with relatively small lengths $c\geq 5.$

\end{document}